\documentclass[12pt]{article}
\usepackage{mystyle}
\begin{document}
\title{Fr\'echet cardinals}
\author{Gabriel Goldberg}

\maketitle
\begin{abstract}
An infinite cardinal \(\lambda\) is called {\it Fr\'echet} if the Fr\'echet filter on \(\lambda\) extends to a countably complete ultrafilter. We examine the relationship between Fr\'echet cardinals and strongly compact cardinals under a hypothesis called the Ultrapower Axiom.
\end{abstract}
\section{Introduction}
An infinite cardinal \(\lambda\) is called {\it Fr\'echet} if the Fr\'echet filter on \(\lambda\) extends to a countably complete ultrafilter. In this paper, we examine the relationship between Fr\'echet cardinals and strongly compact cardinals. Obviously if \(\kappa\) is strongly compact, then every cardinal \(\lambda\) with \(\text{cf}(\lambda) \geq \kappa\) is Fr\'echet. The converse is not provable in ZFC.  Our focus is proving strong converses to this fact under an assumption called the Ultrapower Axiom (UA), which serves as a regularity property for countably complete ultrafilters. Our main theorem is the following:

\begin{thm}[UA]
If \(\delta\) is a Fr\'echet successor cardinal or a Fr\'echet inaccessible cardinal then some \(\kappa\leq \delta\) is \(\delta\)-strongly compact.
\end{thm}

This theorem is a key step in the proof of the equivalence of strong compactness and supercompactness assuming UA (\cite{SC}). The situation for Fr\'echet singular cardinals and weakly inaccessible cardinals is not quite as clear and is wrapped up in the analysis of {\it isolated cardinals}, defined below.

The other main result of this paper uses the analysis of Fr\'echet cardinals to improve the main result of \cite{MO} by removing its cardinal arithmetic hypothesis.

\begin{thm}[UA] The Mitchell order wellorders the class of generalized normal ultrafilters.\end{thm}

We will have to cite a number of results from the author's thesis and from the papers \cite{MO}, \cite{GCH}, and \cite{SC}.

\section{Uniform ultrafilters}
\begin{defn}
An ultrafilter \(U\) on a set \(X\) is {\it Fr\'echet uniform} if for all \(A\in U\), \(|A| = |X|\). A cardinal is {\it Fr\'echet} if it carries a countably complete Fr\'echet uniform ultrafilter.
\end{defn}
Apart from Fr\'echet uniformity, there is another definition of uniformity that is often used. These two notions coincide for ultrafilters on regular cardinals but diverge everywhere else.

\begin{defn}
An ultrafilter \(U\) on an ordinal \(\alpha\) is {\it tail uniform} (or just {\it uniform}) if \(\alpha \setminus \beta\in U\) for all \(\beta < \alpha\). An ordinal is {\it tail uniform} (or just {\it uniform}) if it carries a countably complete uniform ultrafilter.
\end{defn}

The basic relationship between uniform ordinals and Fr\'echet cardinals is quite simple:
\begin{lma}
An ordinal is uniform if and only if its cofinality is Fr\'echet.\qed
\end{lma}

\begin{defn}
For any ordinal \(\alpha\), \(\Un_\alpha\) denotes the set of uniform countably complete ultrafilters on \(\alpha\), \(\Un_{<\alpha} = \bigcup_{\beta <\alpha} \Un_{\beta}\), \(\Un_{\leq\alpha} = \bigcup_{\beta \leq\alpha} \Un_{\beta}\), and \(\Un = \bigcup_{\alpha\in \text{Ord}} \Un_\alpha\).
\end{defn}

Thus \(\alpha\) is uniform if and only if \(\Un_\alpha\neq \emptyset\). Note that if \(\alpha\) is a successor ordinal, then \(\alpha\) is uniform since there is a uniform principal ultrafilter on \(\alpha\).

\begin{defn}
For any \(U\in \Un\), \(\textsc{sp}(U)\) denotes the unique ordinal \(\alpha\) such that \(U\in \Un_\alpha\).
\end{defn}

\section{The Fr\'echet successor operation}
\begin{defn}
For any ordinal \(\gamma\), \(\gamma^\sigma\) denotes the least Fr\'echet cardinal strictly greater than \(\gamma\).
\end{defn}

The following conjecture drives our analysis:

\begin{conj}[UA]\label{IsolationConj}
Suppose \(\gamma\) is an ordinal and \(\lambda = \gamma^\sigma\). Either \(\gamma^\sigma = \gamma^+\) or \(\gamma^\sigma\) is measurable.
\end{conj}

We will verify this conjecture assuming UA + GCH in \cref{UpperBound}, and we will also prove various approximations to it assuming UA alone. But we begin with a related ZFC fact:

\begin{lma}\label{SigmaSuccessor}
For any ordinal \(\gamma\), either \(\gamma^\sigma = \gamma^+\) or \(\gamma^\sigma\) is a limit cardinal.\qed
\end{lma}

This is an immediate consequence of the following lemma.

\begin{lma}\label{SuccessorLemma}
Suppose \(\lambda\) is a cardinal and \(\lambda^+\) is Fr\'echet. Either \(\lambda\) is Fr\'echet or \(\lambda\) is singular and all sufficiently large regular cardinals below \(\lambda\) are Fr\'echet.
\begin{proof}
Fix a countably complete Fr\'echet uniform ultrafilter \(U\) on \(\lambda^+\). 

Assume first that \(\lambda\) is regular. By a theorem of Prikry \cite{Prikry}, \(U\) is \(\lambda\)-decomposable, and therefore \(\lambda\) is Fr\'echet as desired.

Assume instead that \(\lambda\) is singular. Let \(\iota = \text{cf}^{M_U}(\sup j_U[\lambda^+])\). By a theorem of Ketonen \cite{Ketonen}, every set of ordinals \(X\) such that \(|X| \leq \lambda^+\) is contained in a set of ordinals \(X'\in M_U\) such that \(|X'|^{M_U} \leq \iota\). 
\begin{case} \(\iota < \sup j_U[\lambda]\)\end{case}
Fix \(\gamma < \lambda\) such that \(\iota < j_U(\gamma)\). We claim that any regular \(\delta\) with \(\gamma \leq \delta < \lambda\) is Fr\'echet. Note that \(\text{cf}^{M_U}(\sup j_U[\delta]) \leq \iota < j_U(\gamma) < \sup j_U[\delta]\) so \(\sup j_U[\delta]\) is a singular ordinal in \(M_U\). On the other hand by elementarity, \(j_U(\delta)\) is a regular cardinal in \(M_U\). It follows that \(\sup j_U[\delta] < j_U(\delta)\). Therefore \(j_U\) is discontinuous at \(\delta\), and it follows that \(\delta\) is Fr\'echet.

\begin{case} \(\iota\geq\sup j_U[\lambda]\)\end{case}
Note that \(\iota < j_U(\lambda)\) and \(\text{cf}(\iota) = \lambda^+\). Therefore there is a tail uniform (but {\it not} Fr\'echet uniform) ultrafilter \(W\) on the ordinal \(\iota\). Let \(Z = \{X\subseteq \lambda :  j_U(X)\cap \iota\in W\}\). We claim \(Z\) is a Fr\'echet uniform countably complete ultrafilter on \(\lambda\). Suppose \(X\in Z\), and we will show \(|X| = \lambda\). Since \(X\subseteq \lambda\) and \(j_U(X)\cap \iota \in W\), we have that \(j_U(X)\cap \iota\) is cofinal in \(\iota\). Since \(\iota\) is regular in \(M_U\), \(|j_U(X)\cap \iota|^{M_U} = \iota \geq \sup j_U[\lambda]\). It follows easily that \(|X|\geq \lambda\) as desired.
\end{proof}
\end{lma}

\begin{proof}[Proof of \cref{SigmaSuccessor}]
Suppose \(\gamma^\sigma\) is not a limit cardinal. Then \(\gamma^\sigma = \lambda^+\) for some cardinal \(\lambda\). By \cref{SuccessorLemma}, \(\lambda\) is either Fr\'echet or else a limit of Fr\'echet cardinals. Suppose towards a contradiction that \(\gamma < \lambda\). Then there is a Fr\'echet cardinal in the interval \((\gamma,\lambda]\). This contradicts that \(\gamma^\sigma = \lambda^+\).
\end{proof}

We conclude this section by pointing out a consequence of \cref{SuccessorLemma} for \(\omega_1\)-strongly compact cardinals, a notion due to Bagaria-Magidor \cite{MagidorBagaria}:

\begin{cor}
Let \(\kappa\) be the least \(\omega_1\)-strongly compact cardinal. Then \(\kappa\) carries a countably complete Fr\'echet uniform ultrafilter.\qed
\end{cor}

\begin{qst} Let \(\kappa\) be the least \(\omega_1\)-strongly compact cardinal. Does \(\kappa\) carry \(2^{2^\kappa}\) countably complete Fr\'echet uniform ultrafilters?\end{qst}

\section{Approximating ultrafilters}
In this section we exposit two lemmas that allow us to approximate ultrafilters by smaller ultrafilters.

The first lemma is due to the author, but has likely been discovered by others before him.

\begin{lma}
Suppose \(U\) is an ultrafilter on a set \(X\) and \(Y\subseteq j_U(A)\) is a set. Then there is an ultrafilter \(D\) on \(A^Y\) and an elementary embedding \(k : M_D\to M_U\) with \(j_U = k\circ j_D\) and \(Y\subseteq \textnormal{ran}(k)\).
\begin{proof}
Choose for each \(y\in Y\) a function \(f_y : X\to A\) such that \(y = [f_y]_U\). Let \(g : X\to A^Y\) be defined by \(g(x)(y) = f_y(x)\). One calculates that \[[g]_U(j_U(y)) = j_U(g)([\text{id}]_U)(j_U(y)) = j_U(f_y)([\text{id}]_U) = [f_y]_U\] so letting \(D = f_*(U)\) and \(k : M_D\to M_U\) be the factor embedding, \(A\subseteq \text{ran}(k)\) since \([g]_U\) and \(j_U[Y]\) are contained \(\text{ran}(k)\).
\end{proof}
\end{lma}

\begin{cor}\label{Exponential}
Suppose \(U\) is an ultrafilter and \(\gamma\) is an ordinal. Then there is an ultrafilter \(D\) on \(2^\gamma\) and an elementary embedding \(k : M_D\to M_U\) with \(j_U = k\circ j_D\) and \(\gamma\subseteq \textnormal{ran}(k)\).\qed
\end{cor}

The second lemma, due to Silver, is much more interesting. To put it in context, we first spell out a correspondence between partitions modulo an ultrafilter and definability over an ultrapower that is implicit in Silver's proof.

\begin{defn}
Suppose \(P\) is a partition of a set \(X\) and \(A\) is a subset of \(X\). Then the restriction of \(P\) to \(A\) is the partition \(P\restriction A\) defined by \[P\restriction A = \{A\cap S : S\in P\text{ and }A\cap S\neq \emptyset\}\]
\end{defn}
\begin{defn}
Suppose \(U\) is an ultrafilter on a set \(X\). Let \(\mathbb Q_U\) be the preorder on the collection of partitions of \(X\) defined by setting \(P\leq Q\) if there exists some \(A\in U\) such that \(Q\restriction A\) refines \(P\restriction A\). Let \(\mathbb Q^*_U\) be the quotient partial order.
\end{defn}
\begin{defn}
Suppose \(U\) is an ultrafilter. Let \(\mathbb P_U\) be the preorder on \(M_U\) defined by setting \(x \leq y\) if \(x\) is definable over \(M\) from \(y\) and parameters in \(j_U[V]\). Let \(\mathbb P_U^*\) be the quotient partial order.
\end{defn}

\begin{lma}\label{Correspondence}
Suppose \(U\) is an ultrafilter on a set \(X\). Then \(\mathbb P^*_U\cong \mathbb Q^*_U\).
\begin{proof}
It suffices to define an order-embedding \(\Phi : \mathbb P_U\to \mathbb Q_U\) that is {\it essentially surjective} in the sense that for any \(x\in \mathbb P_U\) there is some \(P\in \mathbb Q_U\) such that \(x\) and \(\Phi(P)\) are equivalent in \(\mathbb P_U\).

For \(P\in \mathbb Q_U\), let \(\Phi(P)\) be the unique \(S\in j_U(P)\) such that \([\text{id}]_U\in S\). We claim that \(\Phi: \mathbb P_U\to \mathbb Q_U\) is order-preserving and essentially surjective.

Suppose \(P,Q\in \mathbb Q_U\) and \(P\leq Q\). Fix \(A\in U\) such that \(Q\restriction A\) refines \(P\restriction A\). Then \(\Phi(P)\) is definable in \(M_U\) from the parameters \(\Phi(Q), j_U(P), j_U(A)\) as the unique \(S\in j_U(P)\) such that \(\Phi(Q)\cap j_U(A)\subseteq S\cap j_U(A)\). 

Conversely suppose \(\Phi(P) = j_U(f)(\Phi(Q))\) for some \(f: Q\to P\). Let \(A\subseteq X\) consist of those \(x\in X\) such that \(x\in f(S)\) where \(S\) is the unique element of \(Q\) with \(x\in S\). Then \(A\in U\) since \([\text{id}]_U\in j_U(f)(S)\) where \(S = \Phi(Q)\) is the unique \(S\in j_U(Q)\) such that \([\text{id}]_U\in S\). Moreover for any \(S\in Q\), \(S\cap A\subseteq f(S)\cap A\), so \(Q\restriction A\) refines \(P\restriction A\).

We conclude by showing that \(\Phi\) is essentially surjective. Fix \(x\in \mathbb P_U\). In other words, \(x\in  M_U\), so \(x = j_U(f)([\text{id}]_U)\) for some \(f : X\to V\). Let \[P = \{f^{-1}[\{y\}] : y\in \text{ran}(f)\}\] Then \(\Phi(P)\) is interdefinable with \(x\) over \(M_U\) using parameters in \(j_U[V]\): \(\Phi(P)\) is the unique \(S\in j_U(P)\) such that \(x\in j_U(f)[S]\); and since \(j_U(f)[\Phi(P)] = \{x\}\), \(x = \bigcup j_U(f)[\Phi(P)] \).
\end{proof}
\end{lma}

\begin{defn}
Suppose \(U\) is an ultrafilter on \(X\) and \(\lambda\) is a cardinal. Then \(U\) is {\it \(\lambda\)-indecomposable} if every partition of \(X\) into \(\lambda\) pieces is \(U\)-equivalent to a partition of \(X\) into fewer than \(\lambda\) pieces.
\end{defn}

In other words, \(U\) is \(\lambda\)-indecomposable if and only if there is no Fr\'echet uniform ultrafilter \(W\) on \(\lambda\) with \(W\RK U\).

\begin{thm}[Silver]\label{Silver}
Suppose \(\delta\) is a regular cardinal and \(U\) is an ultrafilter on \(X\) that is \(\lambda\)-indecomposable for all \(\lambda\in [\delta,2^\delta]\). Then there is an ultrafilter \(D\) on some \(\gamma < \delta\) and an elementary embedding \(k : M_D\to M_U\) such that \(j_U = k\circ j_D\) and \(j_U((2^{\delta})^+)\subseteq \textnormal{ran}(k)\).
\begin{proof}
If \(\kappa\) is a cardinal, we call a set \(P\) a {\it \({\leq}\kappa\)-partition} of \(X\) if \(P\) is a partition of \(X\) such that \(|P| \leq \kappa\).

We begin by proving the existence of a maximal \({\leq}2^\delta\)-partition \(P\) of \(X\) in the order \(\mathbb Q^*_U\). Thus we will find a partition \(P\) of \(X\) such that for any refinement \(Q\) of \(P\), there is some \(A\in U\) such that \(P\restriction A\) refines \(Q\restriction A\). 

Suppose there is no such \(P\). We construct by recursion a sequence \(\langle P_\alpha : \alpha \leq \delta\rangle\) of \({\leq}2^\delta\)-partitions of \(X\). Let \(P_0 = \{\delta\}\). If \(\alpha < \delta\) and \(P_\alpha\) has been defined, then let \(P_{\alpha+1}\) be a \({\leq}2^\delta\)-partition of \(X\) witnessing that \(P_\alpha\) is not maximal: thus \(P_{\alpha+1}\) refines \(P_\alpha\) but for any \(A\in U\), \(P_{\alpha}\restriction A\) does not refine \(P_{\alpha+1}\restriction A\); since \(P_{\alpha+1}\) refines \(P_\alpha\), this is the same as saying \(P_\alpha\restriction A \neq P_{\alpha+1}\restriction A\) for all \(A\in U\). If \(\gamma\) a nonzero limit ordinal and \(P_\alpha\) is defined for all \(\alpha < \gamma\), let \(P_\gamma\) be the least common refinement of the \(P_\alpha\) for \(\alpha < \gamma\). 

Since \(U\) is \(\lambda\)-indecomposable for all \(\lambda\in [\delta,2^\delta]\), there is some \(A\in U\) such that \(|P_\delta\restriction A| < \delta\). For \(\alpha \leq \delta\), let \(P_\alpha' = P_\alpha\restriction A\). We claim that the \(P_\alpha'\) are eventually constant. Since \(P_\delta'\) is the least common refinement of the \(P_\alpha'\) and \(\delta\) is regular, it is not hard to show there is some \(\alpha < \delta\) with the property that for each \(S\in P_\alpha'\), there is a unique \(S'\in P_\delta'\) with \(S'\subseteq S\). Then since \(\bigcup P_\alpha' = \bigcup P_\delta' = A\), we must in fact have \(S' = S\). So \(P_\alpha' = P_\delta'\), which implies \(P_\beta' = P_\alpha'\) for all \(\beta \geq \alpha\) as claimed.

This is a contradiction since by our choice of \(P_{\alpha+1}\), \(P_{\alpha+1}\restriction A \neq P_\alpha\restriction A\) for all \(A\in U\). Therefore our assumption was false, and there is a maximal \({\leq}2^\delta\)-partition \(P\) of \(X\). By the indecomposability of \(U\), we may assume \(|P| < \delta\).

Let \(D = \{Q\subseteq P : \bigcup Q\in U\}\). Let \(k : M_D\to M_U\) be the factor embedding. By the maximality of \(P\) among \({\leq}2^\delta\)-partitions and the correspondence of \cref{Correspondence}, \(j_U(2^\delta)\subseteq \text{ran}(k)\). Assume towards a contradiction that there is some \(e < j_U((2^\delta)^+)\) with \(e\notin \text{ran}(k)\). Then the ultrafilter derived from \(j_U\) using \(e\) witnesses that \(U\) is \((2^\delta)^+\)-decomposable. By a theorem of Prikry \cite{Prikry}, letting \(\lambda = \text{cf}(2^\delta)\), \(U\) is \(\lambda\)-decomposable. But \(\lambda\in [\delta,2^\delta]\), which is a contradiction.
\end{proof}
\end{thm}

\section{A weakening of \cref{IsolationConj}}
In this section we prove:

\begin{thm}[UA]\label{UpperBound}
Suppose \(\delta\) is a regular cardinal that is not Fr\'echet. Either \(\delta^\sigma \leq 2^\delta\) or \(\delta^\sigma\) is measurable.
\end{thm}

For the proof, we introduce some definitions we will use throughout this paper.

\begin{defn} Suppose \(U\) is a countably complete ultrafilter and \(W'\in (\Un_{\gamma'})^{M_U}\) for some ordinal \(\gamma'\). Then \(U^-(W') = \{X\subseteq \gamma : j_U(X)\cap \gamma'\in W'\}\) where \(\gamma\) is the least ordinal such that \(j_U(\gamma)\geq \gamma'\).\end{defn}

We briefly recall the definition of the Ketonen order.
\begin{defn}
For \(U,W\in \Un\), we set \(W\sE U\) if setting \(\alpha = [\text{id}]_U\), there is some \(W'\in\Un^{M_U}_{\leq\alpha}\) with \(W = U^-(W')\).
\end{defn}

Proofs of the following theorems appear in \cite{SO}.

\begin{thm}
The Ketonen order is a strict wellfounded partial order of \(\Un\).\qed
\end{thm}

\begin{thm}
The following are equivalent:
\begin{enumerate}[(1)]
\item The Ketonen order is linear
\item The Ultrapower Axiom holds.\qed
\end{enumerate}
\end{thm}

\begin{defn}
If \(\lambda\) is a Fr\'echet cardinal, then \(U_\lambda\) denotes the \(\sE\)-least Fr\'echet uniform countably complete ultrafilter on \(\lambda\).
\end{defn}

The internal relation is a minor variant of the generalized Mitchell order on countably complete ultrafilters.

\begin{defn}
The internal relation \(\I\) is defined on countably complete ultrafilters \(U,W\) by setting \(U\I W\) if \(j_U\restriction M_W\) is an internal ultrapower embedding of \(M_W\).
\end{defn}

We will use the following lemma from \cite{IR}:
\begin{defn}
Suppose \(U,W\in \Un\). Then \(t_U(W)\) denotes the \(\sE\)-least ultrafilter \(W'\in \Un^{M_U}\) such that \(U^-(W') = W\). The function \(t_U :\Un \to \Un^{M_U}\) is called the {\it translation function} associated to \(U\).
\end{defn}

Note that for any \(U,W\in \Un\), \(U^-(j_U(W)) = W\) and so \(t_U(W) \E j_U(W)\). Equality holds if and only if \(U\I W\):

\begin{lma}[UA]\label{InternalTranslation}
For any \(U,W\in \Un\), \(U\I W\) if and only if \(t_U(W) = j_U(W)\).\qed
\end{lma}

The following simple lemma will be refined in \cref{IsolationTheorem} using a much harder argument:

\begin{lma}[UA]\label{SuccessorInternal}
Suppose \(\gamma\) is an ordinal and \(\lambda = \gamma^\sigma\). Then for all \(D\in \Un_{<\lambda}\), \(D\I U_\lambda\).
\begin{proof}
Let \(U = U_\lambda\). It is not hard to see that \(U\) is the \(\sE\)-least countably complete uniform ultrafilter that is not isomorphic to an ultrafilter in \(\Un_{{\leq}\gamma}\). Therefore to prove the lemma, it suffices to show that for all \(D\in \Un_{\leq\gamma}\), \(D\I U\). (This is not really necessary, but it simplifies notation.)

Therefore fix \(D\in \Un_{\leq \gamma}\), and we will show \(D\I U\). Let \(U' = t_D(U)\). It suffices by \cref{InternalTranslation} to show that \(U' = j_D(U)\). Since \(U'\E j_D(U)\), we just need to show that \(U' \not \sE j_D(U)\). 

Assume towards a contradiction that \(U'\sE j_D(U)\). In \(M_D\), \(j_D(U)\) is the \(\sE\)-least countably complete uniform ultrafilter that is not isomorphic to an ultrafilter in \(\Un^{M_D}_{{\leq} j_D(\gamma)}\). Therefore \(U'\) is isomorphic to some \(U''\in \Un^{M_D}_{\leq j_D(\gamma)}\). It follows that there is some \(Z\in \Un_{\leq\gamma}\) such that \(j_Z = j^{M_D}_{U''}\circ j_D = j^{M_D}_{U'}\circ j_D\). (That is, \(Z\) is isomorphic to the sum of \(D\) with \(U'\).)

Since \(D^-(U') = U\), there is an elementary embedding \(k : M_U\to M^{M_D}_{U'}\) such that \(k\circ j_U = j^{M_D}_{U'}\circ j_D\), defined by \(k(j_U(f)([\text{id}]_U)) = j_{U'}^{M_D}(j_D(f))([\text{id}]^{M_D}_{U'})\). It follows that \(U\RK Z\). But a Fr\'echet uniform ultrafilter on \(\lambda\) cannot lie below an ultrafilter on \(\gamma < \lambda\) in the Rudin-Keisler order. (Fix a function \(f : \gamma\to \lambda\) such that \(U = f_*(Z)\); then \(f[\gamma]\in U\), contradicting that \(U\) is Fr\'echet uniform.) This is a contradiction, so the assumption that \(U' \sE j_D(U)\) was false, completing the proof.
\end{proof}
\end{lma}

\cref{UpperBound} will be an immediate consequence of the following lemma:

\begin{lma}\label{UpperBoundLemma}
Suppose \(\delta\) is a regular cardinal that is not Fr\'echet. Let \(\lambda = \delta^\sigma\) and assume \(2^\delta < \lambda\). Suppose \(U\) is a countably complete ultrafilter such that \(D\I U\) for all \(D\in \Un_{<\delta}\). Then \(U\) is \(\lambda\)-complete.
\begin{proof}
There are no Fr\'echet cardinals in the interval \([\delta,2^\delta]\), so in particular \(U\) is \(\lambda\)-indecomposable for all cardinals \(\lambda\in [\delta,2^\delta]\). \cref{Silver} yields some \(D\in \Un_{<\delta}\) such that \(j_D\restriction 2^\delta = j_U\restriction 2^\delta\). Since \(D\I U\), \( j_U\restriction 2^\delta = j_D\restriction 2^\delta\in M_U\). Thus \(U\) is \(2^\delta\)-supercompact. 

As a consequence of the Kunen inconsistency theorem \cite{Kunen}, either \(j_U\) is discontinuous at every regular cardinal in \([\delta,2^\delta]\) or \(U\) is \(2^\delta\)-complete.  The former cannot hold since there are no Fr\'echet cardinals in the interval \([\delta,2^\delta]\). Therefore \(U\) is \(2^\delta\)-complete. Let \(\kappa\) be the completeness of \(U\). Since \(\kappa\) is measurable, \(\kappa\) is Fr\'echet, and so since \(\kappa > \delta\), \(\kappa \geq \delta^\sigma = \lambda\). Thus \(U\) is \(\lambda\)-complete.
\end{proof}
\end{lma}

\begin{proof}[Proof of \cref{UpperBound}]
Let \(\lambda = \delta^\sigma\) and assume that \(\lambda > 2^\delta\). Let \(U = U_\lambda\). By \cref{SuccessorInternal}, \(D\I U\) for all \(D\in \Un_{<\delta}\). By \cref{UpperBoundLemma}, \(U\) is \(\lambda\)-complete. Therefore \(U\) is a \(\lambda\)-complete ultrafilter on \(\lambda\), so \(\lambda\) is measurable.
\end{proof}

\section{Nonisolation lemmas}
\begin{lma}[UA]\label{Nonisolation}
Suppose \(\lambda\) is a limit cardinal. Suppose there is some \(W\in \Un\) such that \(j_W\) is discontinuous at \(\lambda\) and \(U_\lambda\I W\). Then \(\lambda\) is a limit of Fr\'echet cardinals.
\end{lma}

For the proof, we need another fact about the internal relation, a sort of dual to \cref{InternalTranslation}.

\begin{defn}
Suppose \(U\in \Un_\lambda\) and \(W\) is a countably complete ultrafilter. Then \(s_W(U) = \{X\subseteq \lambda_* : X\in M_U\text{ and }j_W^{-1}[X]\in U\}\) where \(\lambda_* = \sup j_W[\lambda]\).
\end{defn}

\begin{lma}
If \(U,W\in \Un\) then \(j^{M_W}_{s_W(U)} = j_U\restriction M_W\).\qed
\end{lma}

\begin{cor}
If \(U,W\in \Un\), \(U\I W\) if and only if \(s_W(U)\in M_W\).\qed
\end{cor}

We also need a standard fact about covering in ultrapowers.

\begin{lma}\label{jCover}
Suppose \(j :  V\to M\) is an ultrapower embedding. Suppose \(\lambda\) is a cardinal and \(\lambda'\) is an \(M\)-cardinal. Assume there is a set \(X\subseteq M\) with \(|X| = \lambda\) such that \(j[X]\subseteq X'\) for some \(X'\in M\) such that \(|X'|^M \leq \lambda'\). Then any set \(A\subseteq M\) with \(|A| = \lambda\) is contained in a set \(A' \in M\) such that \(|A'|^M\leq \lambda'\).
\begin{proof}
Fix \(a\in M\) such that every element of \(M\) is of the form \(j(f)(a)\) for some function \(f\). Choose functions \(\langle f_x : x \in X\rangle\) such that \(A = \{j(f_x)(a) : x\in X\}\). Let \(\langle g_x : x\in j(X)\rangle = j(\langle f_x : x \in X\rangle)\). Let \(A' = \{g_x(a) : x\in X'\}\). Easily \(A'\in M\), \(|A'|^M \leq |X'|^M \leq \lambda'\), and \(A\subseteq A'\), as desired.
\end{proof}
\end{lma}

\begin{proof}[Proof of \cref{Nonisolation}]
Let \(U = U_\lambda\). Let \(\lambda' = \sup j_W[\lambda]\) and let \(U' = s_W(U)\).
\begin{case}
\(U'\) is Fr\'echet uniform in \(M_W\).
\end{case}
In this case \(U'\) witnesses that \(\lambda'\) is Fr\'echet in \(M_W\). Since \(j_W\) is discontinuous at \(\lambda\), a simple reflection argument implies that \(\lambda\) is a limit of Fr\'echet cardinals.
\begin{case}
\(U'\) is not Fr\'echet uniform in \(M_W\).
\end{case}
Fix \(X'\in U'\) and \(\gamma < \lambda\) such that \(|X|^{M_W} < j_W(\gamma)\). Let \(X = j_W^{-1}[X']\). Then \(X\in U\), so since \(U\) is Fr\'echet uniform, \(|X| = \lambda\). But \(j_W[X]\subseteq X'\). It follows from \cref{jCover} that every set \(A\subseteq M_W\) with \(|A| \leq \lambda\) is covered by a set \(A'\in M_W\) with \(|A'|^{M_W} \leq \lambda'\).

It follows that \(j_W\) is discontinuous at every regular cardinal in the interval \([\gamma,\lambda]\). (This is a standard consequence of \((\gamma,\lambda)\)-regularity, which is what we have established. If \(\delta\in [\gamma,\lambda]\) is a regular cardinal, then \(\text{cf}^{M_W}(\sup j_W[\delta]) < j_W(\gamma)\) by the covering property of \(M_W\), so \(j_W(\delta)\), being regular in \(M_W\), is not equal to \(\sup j_W[\delta]\).) Therefore \(\lambda\) is a limit of Fr\'echet cardinals.
\end{proof}

As a corollary of the proof we also have the following fact:
\begin{lma}[UA]\label{Nonisolation2}
Suppose \(\lambda\) is a Fr\'echet limit cardinal and for some countably complete \(W\), \(U_\lambda \I W\) and \(W\not \I U_\lambda\). Then \(\lambda\) is a limit of Fr\'echet cardinals.
\begin{proof}
Let \(U = U_\lambda\). We may assume that \(j_W\) is continuous at \(\lambda\). 
\begin{clm}
\(t_W(U)\) is not Fr\'echet uniform in \(M_W\).
\end{clm}
\begin{proof} Assume towards a contradiction that \(t_W(U)\) is Fr\'echet uniform in \(M_W\). By the definition of \(t_W(U)\), \(t_W(U) \E^{M_W} j_W(U)\). By elementarity \(j_W(U)\) is the \(\sE^{M_W}\)-least Fr\'echet uniform ultrafilter on \(j_W(\lambda) = \sup j_W[\lambda]\) in \(M_W\). Since \(\sup j_W[\lambda]\leq\textsc{sp}(t_W(U))\), this implies \(j_W(U)\E^{M_W} t_W(U)\). It follows that \(t_W(U) = j_W(U)\). But by \cref{InternalTranslation}, \(W\I U\), contrary to hypothesis.\end{proof}

Therefore there is some \(X'\in t_W(U)\) such that \(|X'|^{M_W} < j_W(\gamma)\) for some \(\gamma < \lambda\). We now proceed as in \cref{Nonisolation} to show all sufficiently large regular cardinals below \(\lambda\) are Fr\'echet.
\end{proof}
\end{lma}

\section{The strong compactness of \(\kappa_{U_\delta}\)}
The main theorem of this section is the key to the supercompactness analysis of \cite{SC}.
\begin{thm}[UA]\label{MainThm1}
If \(\delta\) is a Fr\'echet successor cardinal or a Fr\'echet inaccessible cardinal, then \(\kappa_{U_\delta}\) is \(\delta\)-strongly compact.
\end{thm}

The theorem requires a sequence of lemmas. The first is related to the phenomenon of commuting ultrapowers first discovered by Kunen.

\begin{lma}[UA]\label{Commutativity}
Suppose \(U\) and \(W\) are countably complete ultrafilters. The following are equivalent:
\begin{enumerate}[(1)]
\item \(U\I W\) and \(W\I U\).
\item \(j_U(j_W) = j_W\restriction M_U\).
\item \(j_W(j_U) = j_U\restriction M_W\).\qed
\end{enumerate}
\end{lma}

The proof appears in \cite{IR}.

We also need a consequence of \cref{Exponential}.
\begin{lma}[UA]\label{Completeness}
Suppose \(\kappa\) is a strong limit cardinal and \(W\) is a countably complete ultrafilter. The following are equivalent:
\begin{enumerate}[(1)]
\item \(j_W[\kappa]\subseteq \kappa\) and for all \(D\in \Un_{<\kappa}\), \(D\I W\).
\item \(W\) is \(\kappa\)-complete.
\end{enumerate}
\begin{proof}
We will prove (1) implies (2).

We claim \(j_W[\alpha]\in M_W\) for all \(\alpha < \kappa\). Let \(\alpha' = \sup j_W[\alpha]\), and note that \(\alpha' < \kappa\), so \(2^{\alpha'} < \kappa\). By \cref{Exponential}, there is a countably complete ultrafilter \(D\) on \(2^{\alpha'}\) such that \(j_D[\alpha] = j_W[\alpha]\). By assumption \(D\I W\), so \(j_D[\alpha]\in M_W\). Hence \(j_W[\alpha]\in M_W\).

By the Kunen inconsistency theorem there cannot be a \(j : V\to M\) such that \(j[\lambda]\in M\) for some \(\lambda\) with \(j[\lambda]\subseteq\lambda\) and \(\textsc{crt}(j) < \lambda\). Taking \(\lambda = \kappa\) and \(j = j_W\), it follows that \(W\) is \(\kappa\)-complete.
\end{proof}
\end{lma}

This lemma has the following corollary.

\begin{lma}[UA]\label{Noncommutative}
Suppose \(U\) and \(W\) are countably complete ultrafilters such that \(\Un_{<\kappa_U}\I W\) and \(\Un_{<\kappa_W}\I U\). Either \(U\not \I W\) or \(W\not \I U\).
\begin{proof}
Suppose not. Then \(U\I W\) and \(W\I U\). By \cref{Commutativity}, \(j_U(j_W) = j_W\restriction M_U\) and \(j_W(j_U) = j_U\restriction M_W\). In particular, \(j_U(\kappa_W) = \kappa_W\) and \(j_W(\kappa_U) = \kappa_U\). Therefore \(\kappa_U \neq \kappa_W\), and we may assume by symmetry that \(\kappa_U < \kappa_W\). But then \(j_U[\kappa_W]\subseteq \kappa_W\) and \(\Un_{<\kappa_W}\I U\). So \(U\) is \(\kappa_W\)-complete by \cref{Completeness}. But this contradicts that \(\kappa_U < \kappa_W\).
\end{proof}
\end{lma}

\begin{lma}[UA]\label{Nonisolation3}
Suppose \(W\) is a countably complete ultrafilter and \(\nu\) is an isolated cardinal such that \(U_\nu\I W\) and for all \(D\in \Un_{<\nu}\), \(D\I W\). Then \(W\) is \(\nu^+\)-complete.
\end{lma}
\begin{proof}
Suppose towards a contradiction \(\kappa_W\leq \nu\). Then \(\Un_{<\kappa_W}\I U_\nu\) by \cref{SuccessorInternal}. By assumption \(\Un_{<\kappa_U}\I W\). Therefore by \cref{Noncommutative}, \(W\not \I U_\nu\). But by \cref{Nonisolation2}, it follows that \(\nu\) is not isolated, and this is a contradiction.
\end{proof}

\begin{thm}[UA]\label{SuccessorThm}
Suppose \(\lambda\) is a Fr\'echet cardinal that is either regular or isolated. For any \(\gamma\) with \(\kappa_{U_\lambda}\leq \gamma < \lambda\), either \(\gamma^\sigma = \gamma^+\) or \(\gamma^\sigma = \lambda\).
\begin{proof}
Let \(\nu = \gamma^\sigma\). Assume \(\nu < \lambda\). Assume towards a contradiction that \(\nu \neq \gamma^+\). By \cref{SigmaSuccessor}, \(\nu\) is a limit cardinal. So \(\nu\) is isolated.. Since \(\lambda\) is either regular or isolated, we have \(U_\nu \I U_\lambda\) and for all \(D\in \Un_{<\nu}\), \(D\I U_\lambda\). But by \cref{Nonisolation3}, it follows that \(U_\lambda\) is \(\nu^+\)-complete, contradicting that \(\nu \geq \gamma\geq \kappa_{U_\lambda}\). 
\end{proof}
\end{thm}

For the next theorem we need a fact that is essentially due to Ketonen \cite{Ketonen}.
\begin{lma}[UA]\label{Ketothing}
Suppose \(\delta\) is a regular Fr\'echet cardinal and every regular cardinal in the interval \([\kappa_{U_\delta},\delta]\) is Fr\'echet. Then \(\kappa_{U_\delta}\) is \(\delta\)-strongly compact.
\begin{proof}
Let \(U = U_\delta\). Since \(\sup j_{U}[\delta]\) is not tail uniform in \(M_{U}\), \(\text{cf}^{M_{U}}(\sup j_{U}[\delta])\) is not Fr\'echet in \(M_U\). Therefore we have \(\text{cf}^{M_{U}}(\sup j_{U}[\delta]) < j_{U}(\kappa_{U})\), since otherwise \(\text{cf}^{M_{U}}(\sup j_{U}[\delta])\) is an \(M_{U}\)-regular cardinal in the interval \(j_{U}([\kappa_{U},\delta])\) and hence is Fr\'echet by elementarity. This implies that every set of ordinals of size \(\delta\) is covered by a set of ordinals in \(M_U\) of size less than \(j_U(\kappa_U)\). Hence \(j_U: V\to M_U\) witnesses that \(\kappa_U\) is \(\delta\)-strongly compact.
\end{proof}
\end{lma}

\begin{thm}[UA]\label{StrongCompactness}
Suppose \(\delta\) is Fr\'echet cardinal. Assume \(\delta\) is either a successor cardinal or a regular limit of Fr\'echet cardinals. Then \(\kappa_{U_\delta}\) is \(\delta\)-strongly compact.
\begin{proof}
Fix  \(\gamma\) with \(\kappa_{U_\lambda}\leq \gamma^+ < \delta\). We claim \(\gamma^\sigma < \delta\). If \(\delta\) is a limit of Fr\'echet cardinals this is immediate. Suppose instead that \(\delta = \lambda^+\). By \cref{SuccessorLemma}, either \(\lambda\) is Fr\'echet or \(\lambda\) is a limit of Fr\'echet cardinals. Therefore since \(\gamma < \lambda\), \(\gamma^\sigma \leq \lambda < \delta\), as claimed.

It follows that every successor cardinal in the interval \([\kappa_{U_\delta},\delta]\) is Fr\'echet. But then by Prikry's theorem \cite{Prikry}, every regular cardinal in the interval \([\kappa_{U_\delta},\delta]\) is Fr\'echet. By \cref{Ketothing}, it follows that \(\kappa_{U_\delta}\) is \(\delta\)-strongly compact.
\end{proof}
\end{thm}

\begin{proof}[Proof of \cref{MainThm1}]
By \cref{StrongCompactness}, we have reduced to the case that \(\delta\) is a strongly inaccessible isolated cardinal. By \cref{UpperBound}, \(\delta\) is measurable. The proof will be complete if we show \(\kappa_{U_\delta} = \delta\), since \(\delta\) is certainly \(\delta\)-strongly compact. Let \(W\) be a normal ultrafilter on \(\delta\). We claim \(U_\delta = W\). Otherwise \(U_\delta \sE W\), and so since \(W\) is normal, \(U_\delta \mo W\). This implies \(\delta\) is a limit of measurable cardinals, which contradicts that \(\delta\) is isolated.
\end{proof}

\section{Ultrafilters on an isolated cardinal}\label{IsolatedUltrafilters}
In this section we enact a fairly complete analysis of the ultrafilters that lie on an isolated cardinal. {\it For the rest of the section, \(\lambda\) will denote a fixed isolated cardinal and \(U\) will denote \(U_\lambda\).}

We start with a characterization of \(U\).

\begin{thm}[UA]\label{IsolatedUnique}
The ultrafilter \(U\) is the unique countably complete uniform ultrafilter on \(\lambda\) such that \([\textnormal{id}]_U\) is a generator of \(j_U\).
\begin{proof}
Suppose towards a contradiction \(U'\) is the \(\sE\)-least countably complete uniform ultrafilter on \(\lambda\) such that \([\textnormal{id}]_{U'}\) is a generator of \(j_{U'}\) and \(U'\neq U\). Since \([\text{id}]_{U'}\) is a generator of \(U'\), \(U'\) is Fr\'echet uniform. Therefore by the definition of \(U\), \(U\sE U'\). 

Let \((i,i') : (M_U,M_{U'})\to N\) be a comparison of \((j_U,j_{U'})\). Then \(i'([\text{id}]_{U'})\) is a generator of \(i'\circ j_{U'} = i\circ j_U\). Since \(i([\text{id}]_U) < i'([\text{id}]_{U'})\) follows that \(i'([\text{id}]_{U'})\) is an  \(i([\text{id}]_U)\)-generator of \(i\circ j_U\) and hence \(i'([\text{id}]_{U'})\) is a generator of \(i : M_U\to N\). 

Let \(W\) be the ultrafilter derived from \(i\) using \(i'([\text{id}]_{U'})\). Then \(W\) is Fr\'echet uniform. Moreover since \(i'([\text{id}]_U) \geq \sup i'\circ j_{U'}[\lambda] = \sup i\circ j_U[\lambda]\), \(\textsc{sp}(W) \geq \sup j_U[\lambda]\).

 If \(\textsc{sp}(W) < j_U(\lambda)\), then an easy reflection argument yields that \(\lambda\) is a limit of Fr\'echet cardinals, contradicting that \(\lambda\) is isolated. So \(\textsc{sp}(W) = j_U(\lambda)\). Therefore in \(M_U\), \(W\) is a countably complete uniform ultrafilter on \(j_U(\lambda)\) such that \([\textnormal{id}]^{M_U}_{W}\) is a generator of \(j^{M_U}_W\). Moreover \(U^-(W) = U'\) so \(W\neq j_U(U)\). It follows that \(j_U(U')\E^{M_U} W\).
 
But \(W = t_U(U')\). Since \(t_U(U') \E^{M_U} j_U(U')\), it follows that \(t_U(U') = j_U(U')\). This implies \(U\I U'\). But this contradicts \cref{Nonisolation}.
\end{proof}
\end{thm}

We then analyze the generators of \(U^n\).

\begin{defn}
We define finite sets of ordinals \(p_n\) by induction setting \(p_0 = \emptyset\) and \[p_{n+1} = j_{U}(p_n)\cup \{j_{U}(j_{U^n})([\text{id}]_U)\}\]
\end{defn}

\begin{lma}[UA]\label{Generators}
\(j_U(j_{U^n})([\textnormal{id}]_U)\) is a \(j_U(p_n)\)-generator of \(j_{U^{n+1}}\).
\end{lma}

For the proof we need an abstract version of the Dodd-Jensen lemma:
\begin{prp}\label{MinDefEmb}
Suppose \(M\) and \(N\) are inner models and \(j : M\to N\) and \(k : M\to N\) are elementary embeddings such that \(j\) is definable from parameters over \(M\). Then \(j(\alpha) \leq k(\alpha)\) for any ordinal \(\alpha\).\qed
\end{prp}

This yields the following fact:

\begin{lma}[UA]\label{GenLemma}
Suppose \(Z\) is a countably complete ultrafilter and \(h : M_U\to M_Z\) is an internal ultrapower embedding with \(j_Z = h\circ j_U\). Then \(h([\textnormal{id}]_U)\) is the least generator of \(M_Z\) above \(\sup j_Z[\lambda]\).
\begin{proof}
Let \(\xi\) be the least generator of \(j_Z\) above \(\sup j_Z[\lambda]\), so \(h([\text{id}]_U) \leq \xi\). Let \(U'\) be the ultrafilter derived from \(Z\) using \(\xi\) and let \(k : M_{U'}\to M_Z\) be the canonical factor embedding. By \cref{IsolatedUnique}, \(U' = U\). By \cref{MinDefEmb}, \(\xi = k([\text{id}]_U) \leq h([\text{id}]_U)\). Thus \(h([\text{id}]_U) = \xi\), as desired.
\end{proof}
\end{lma}

\begin{proof}[Proof of \cref{Generators}]
Suppose not. Let \(p \subseteq j_U(j_{U^n})([\textnormal{id}]_U)\) be the least parameter such that \[j_U(j_{U^n})([\textnormal{id}]_U)\in H^{M_{U^{n+1}}}(j_{U^{n+1}}[V]\cup j_U(p_n)\cup p)\]

\begin{case} \(p \subseteq \sup j_{U^{n+1}}[\lambda]\). \end{case}
Note that \(s_{U^n}(U)\) is the ultrafilter derived from \(j_U\restriction M_{U^n}\) using \(j_U(j_{U^n})([\textnormal{id}]_U)\). It follows that \(s_{U^n}(U)\) is isomorphic over \(M_{U^n}\) to the uniform \(M_{U^n}\)-ultrafilter \(U'\) derived from \(j_U\restriction M_{U^n}\) using \(p\). In other words, for some ordinal \(\delta < \sup j_{U^{n+1}}[\lambda]\), there is some \(f :[\delta]^{<\omega} \to \sup j_{U^n}[\lambda]\) with \(f\in M_{U^n}\) such that \(X\in s_{U^n}(U)\) if and only if \(f^{-1}[X]\in U'\). 

In particular, the set \(X = f[[\delta]^{<\omega}]\) is such that \(X\in s_{U^n}(U)\) and \(|X|^{M_{U^n}} \leq \delta\). Let \(\bar X = j_{U^n}^{-1}[X]\). Then \(\bar X\in U\), so \(|\bar X| = \lambda\). But by \cref{jCover}, this implies every set \(A\subseteq M_{U^n}\) with \(|A|\leq \lambda\) is covered by a set \(A'\in M_{U^n}\) with \(|A'|^{M_{U^n}}\leq \delta\). As in \cref{Nonisolation}, it follows that \(j_{U^n}\) is discontinuous at all sufficiently large regular cardinals below \(\lambda\), which contradicts that \(\lambda\) is isolated.
\begin{case} \(p\not\subseteq \sup j_{U^{n+1}}[\lambda]\). \end{case}
Since \(p\) is a set of generators of \(j_{U^{n+1}}\), it follows that there is a generator of \(j_{U^{n+1}}\) in the interval \([\sup j_{U^{n+1}}[\lambda], j_U(j_{U^n})([\textnormal{id}]_U))\). But this contradicts \cref{GenLemma} with \(Z = U^{n+1}\) and \(h = j_U(j_{U^n})\).
\end{proof}

\begin{cor}[UA]
\(p_n\) is the least parameter \(p\) with \(M_{U^n} = H^{M_{U^n}}(j_{U^n}[V]\cup p)\).\qed
\end{cor}

\begin{defn}
If \(W\) is a countably complete ultrafilter on \(\lambda\), then \(p_W\) denotes the least parameter \(p\) with \(M_W = H^{M_W}(j_W[V]\cup p\cup q)\) for some \(q\subseteq \sup j_W[\lambda]\).
\end{defn}

Thus \(p_W\) is the piece of the Dodd parameter of \(W\) contained in the interval \( [\sup j_W[\lambda],j_W(\lambda))\).

\begin{thm}[UA]\label{IsolationTheorem}
Suppose \(W\) is a countably complete ultrafilter on \(\lambda\). Let \(n = |p_W|\) and let \(\lambda_* = \sup j_{U^n}[\lambda]\). There is some \(D\in \Un^{M_{U^n}}_{<\lambda_*}\) with \(j^{M_{U^n}}_D\circ j_{U^n} = j_W\) and \(j^{M_{U^n}}_D(p_n) = p_W\).
\begin{proof}
We may assume by induction that the theorem is true for all \(W' \sE W\). We may also assume without loss of generality that \(W\) is \(\lambda\)-decomposable. It follows that \(U\not \I W\) by \cref{Nonisolation}.

Let \((h,i): (M_U,M_W)\to N\) be the canonical comparison of \((j_U,j_W)\). Let \(p = i(p_W)\) and let \(\xi = [\text{id}]_U\). 
\begin{clm}\label{RFClm} \(h(\xi) = \min p\).\end{clm}
\begin{proof}
Suppose not. Then by \cref{GenLemma}, \(h(\xi) < \min p\). 
\begin{sclm}
\(h(\xi)\) is a \(p\)-generator of \(h\circ j_U\).
\begin{proof}
We have that \(h = j^{M_U}_{W'}\) where \(W' = t_U(W)\). Since \(U\not \I W\), we must have \(t_U(W) \sE^{M_U} j_U(W)\). It is easy to check that \(p = (p_{W'})^{M_U}\) where the parameter is defined at \(j_U(\lambda)\) rather than at \(\lambda\). Therefore by our induction hypothesis and elementarity, there is some \(n < \omega\) such that \(h = k\circ j_U(j_{U^n})\) where \(k : M^{M_U}_{j_U(U^n)}\to N\) is an elementary embedding such that \(k(j_U(p_n)) = p\). By \cref{Generators}, \(j_U(j_{U^n})(\xi)\) is a \(j_U(p_n)\)-generator of \(j_{U^{n+1}} =j_U(j_{U^n}) \circ j_U\). Therefore \(k(j_U(j_{U^n})(\xi))\) is a \(k(j_U(p_n))\)-generator of \(k\circ j_U(j_{U^n})\circ j_U\). Replacing like terms, \(h(\xi)\) is a \(p\)-generator of \(h\circ j_U\).
\end{proof}
\end{sclm}
It follows that \(h(\xi)\) is a generator of the embedding \(i : M_W\to N\). Let \(Z\) be the ultrafilter derived from \(i\) using \(h(\xi)\). Then \(\textsc{sp}(Z)\geq \sup j_W[\lambda]\) since \(h(\xi)\geq \sup i\circ j_W[\lambda]\). On the other hand since \(h(\xi) < \min p\), \(\textsc{sp}(Z) \leq \min p < j_W(\lambda)\). But in \(M_W\), \(Z\) is a Fr\'echet uniform ultrafilter, since \([\text{id}]_Z\) is a generator of \(j^{M_W}_Z\). This contradicts that by isolation \(j_W(\lambda) = (\gamma^\sigma)^{M_W}\) for some \(\gamma < \sup j_W[\lambda]\). This contradiction proves \cref{RFClm}.
\end{proof}
By the definition of a canonical comparison, it follows that \(i\) is the identity, \(N = M_W\), and \(h : M_U\to M_W\) is an internal ultrapower embedding. Moreover, as in the proof of \cref{RFClm}, our induction hypothesis implies that for some \(n < \omega\), \[h = j^{M_{U^{n+1}}}_D\circ j_U(j_{U^n})\] where \(D\in \Un_{< \lambda_*}\) for \(\lambda_*  = \sup j_U(j_U^n)[j_U(\lambda)]\) and \(j^{M_{U^{n+1}}}_D(j_U(p_n)) = p \setminus (h(\xi) + 1)\). Since \(j_{U^{n+1}}(\lambda) = (\gamma^\sigma)^{M_{U^{n+1}}}\), there is in fact some \(D'\in \Un^{M_{U^{n+1}}}_{<\sup j_{U^{n+1}}[\lambda]}\) with \(D'\equiv D\) in \(M_{U^{n+1}}\). Thus replacing \(D\) with \(D'\), we may assume \(D\in \Un^{M_{U^{n+1}}}_{<\sup j_{U^{n+1}}[\lambda]}\).

Since \(p_{n+1} = j_U(p_n)\cup j_U(j_{U^n})(\xi)\), we have \(j^{M_{U^{n+1}}}_D(p_{n+1}) = p \setminus (h(\xi) + 1)\cup \{h(\xi)\} = p\), with the final equality following from \cref{RFClm}.

Putting everything together, we have shown that there is some \(D\in \Un^{M_{U^{n+1}}}_{<\sup j_{U^{n+1}}[\lambda]}\) such that \(j_W = j^{M_{U^{n+1}}}_D\circ j_{U^{n+1}}\) and such that \(j^{M_{U^{n+1}}}_D(p_{n+1}) = p_W\). This proves the theorem.
\end{proof}
\end{thm}

We will use the following consequence of this theorem, which improves \cref{SuccessorInternal}.

\begin{cor}[UA]\label{IsolatedInternal}
Suppose \(D\) is a countably complete \(M_U\)-ultrafilter on an ordinal below \(\sup j_U[\lambda]\). Then \(D\in M_U\).
\begin{proof}
Let \(W\) be a countably complete ultrafilter on \(\lambda\) such that \(j_W = j^{M_U}_D\circ j_U\). An easy calculation shows that \(p_W = \{j^{M_U}_D([\text{id}]_U)\}\). Therefore by \cref{IsolationTheorem}, there is an internal ultrapower embedding \(k : M_U\to M_W\) such that \(j_W = k\circ j_U\) and \(k([\text{id}]_U) = j^{M_U}_D([\text{id}]_U)\). It follows that \(k = j^{M_U}_D\). But then \(j^{M_U}_D\) is an internal ultrapower embedding, so \(D\in M_U\).
\end{proof}
\end{cor}

\section{The continuum function below an isolated cardinal}
In this section we prove a theorem that shows \cref{UpperBound} is optimal in a sense: 

\begin{thm}[UA]\label{LowerBound}
Suppose \(\lambda\) is an isolated cardinal and \(\gamma < \lambda\) is Fr\'echet. Then \(2^\gamma < \lambda\).
\end{thm}

We need a theorem from \cite{SC}:

\begin{thm}[UA]\label{InternalSuper}
If \(W\) is a countably complete ultrafilter and \(\lambda\) is a cardinal such that \(\Un_{<\lambda}\I W\), then for all Fr\'echet cardinals \(\gamma < \lambda\), \(W\) is \(\gamma\)-supercompact.
\begin{proof}
Suppose \(\delta\) is the least cardinal such that \(M_W\) is not closed under \(\delta\)-sequences. If \(\lambda\leq \delta\), we are done, so assume \(\delta < \lambda\). We must show that no cardinal in the interval \([\delta,\lambda]\) is Fr\'echet. 

\begin{clm} \(\delta\) is not Fr\'echet. 
\begin{proof}
Assume towards a contradiction that \(\delta\) is Fr\'echet. Note that \(\delta\) is a regular cardinal. We claim \(\delta\) is not isolated. This is because if \(\delta\) were isolated, then since \(\Un_{<\delta}\I W\), by \cref{Nonisolation3}, \(W\) would be \(\delta^+\)-complete, contradicting that \(M_W\) is not closed under \(\delta\)-sequences. Thus \(\delta\) is not isolated. 

By \cref{StrongCompactness}, it follows that \(\textsc{crt}(j_\delta)\) is \(\delta\)-strongly compact, so by the main theorem of \cite{SC}, \(M_\delta\) is closed under \({<}\delta\)-sequences and has the tight covering property at \(\delta\). Now \(j_\delta(M_W)\subseteq M_W\) and \(j_\delta(M_W) = M_{j_\delta(W)}^{M_\delta}\) is closed under \({<}j_\delta(\delta)\)-sequences inside of \(M_\delta\). In particular, since \(\delta < j_\delta(\delta)\), \begin{equation}\label{ClosureW}[\text{Ord}]^\delta\cap M_\delta\subseteq M_W\end{equation}

Fix a set of ordinals \(S\) of cardinality \(\delta\). By the tight covering property at \(\delta\), there is some \(A\in M_\delta\) containing \(j_\delta[S]\) such that \(|A|^{M_\delta} = \delta\). Let \(A' = A\cap j_\delta(S)\), so that \(j_\delta^{-1}[A'] =S\). By \cref{ClosureW},  \(A'\in M_W\). Since \(U_\delta\I W\), \(j_\delta\restriction M_W\) is definable over \(M_W\). Therefore \(j_\delta^{-1}[A']\in M_W\), so \(S\in M_W\). It follows that \(M_W\) is closed under \(\delta\)-sequences, contradicting the definition of \(\delta\).
\end{proof}
\end{clm}
To show no cardinal in the interval \([\delta,\lambda]\) is Fr\'echet, it now suffices to show that \(\delta^\sigma \geq \lambda\) so \(\lambda\) is a strong limit cardinal. Since \(\delta\) is regular but not Fr\'echet, \(\delta^\sigma > \delta^+\). Let \(\nu = \delta^\sigma\). Assume towards a contradiction that \(\nu < \lambda\). Then \(U_\nu \I W\) and \(\Un_{<\nu} \I W\). Therefore \(W\) is \(\nu^+\)-complete, contradicting that \(M_W\) is not closed under \(\nu\)-sequences.
\end{proof}
\end{thm}

We need the following lemma whose proof requires ideas from \cite{GCH}:

\begin{lma}[UA]\label{LimitUniformStrong}
Suppose \(\lambda\) is a limit of Fr\'echet cardinals. Then \(\lambda\) is a strong limit cardinal.
\begin{proof}[Sketch]
Let \(\kappa\) be the supremum of the isolated cardinals below \(\lambda\). If \(\kappa = \lambda\) then it is easy to see that \(\lambda\) is a limit of measurable cardinals so \(\lambda\) is a strong limit cardinal. Otherwise by \cref{SuccessorThm}, every successor cardinal in the interval \([\kappa,\lambda]\) is Fr\'echet. Applying \cref{InternalSuper} to \(U_\delta\) for successor cardinals \(\delta\in [\kappa,\lambda]\), it follows that \(\kappa\) is \(\delta\)-supercompact for all \(\delta < \lambda\). Now by the main theorem of \cite{GCH}, GCH holds for every cardinal in the interval \([\kappa,\lambda]\). So \(\lambda\) is a strong limit cardinal.
\end{proof}
\end{lma}

\begin{proof}[Proof of \cref{LowerBound}]
Assume towards a contradiction that \(\lambda\) is the least isolated cardinal such that for some Fr\'echet cardinal \(\gamma < \lambda\), \(2^\gamma \geq \lambda\). Since \(U_\lambda\) satisfies the hypotheses of \cref{InternalSuper}, \(U_\lambda\) is \(\gamma\)-supercompact. In particular, \(P(\gamma)\subseteq M_{U_\lambda}\), and therefore \((2^\gamma)^{M_{U_\lambda}}\geq 2^\gamma\geq  \lambda\).
\begin{clm}
\(\lambda < j_{U_\lambda}(\kappa_{U_\lambda})\).
\begin{proof}
Since \( j_{U_\lambda}(\kappa_{U_\lambda})\) is inaccessible in \(M_{U_\lambda}\) and \((2^\gamma)^{M_{U_\lambda}}\geq \lambda\), it suffices to show \(\gamma < j_{U_\lambda}(\kappa_{U_\lambda})\). If \(j_{U_\lambda}(\kappa_{U_\lambda}) < \gamma\), then \(j_{U_\lambda}\) witnesses that \(\kappa_{U_\lambda}\) is a huge cardinal. In particular, some \(\kappa < \kappa_{U_\lambda}\) is \(\kappa_{U_\lambda}\)-supercompact. A standard theorem on the propagation of supercompactness (see \cite{Kanamori}) implies \(\kappa\) is \(\gamma\)-supercompact. 

By the Kunen inconsistency theorem, there is an inaccessible cardinal \(\delta \leq \gamma\) such that \(j_{U_\lambda}(\delta) > \gamma\). By elementarity, in \(M_{U_\lambda}\), \(\kappa\) is \(j_{U_\lambda}(\delta)\)-supercompact. By the results of \cite{GCH} applied in \(M_{U_\lambda}\), \((2^\gamma)^{M_{U_\lambda}} = \gamma^+\), which contradicts the fact that \((2^\gamma)^{M_{U_\lambda}}\geq \lambda\).
\end{proof}
\end{clm}

By \cref{IsolatedInternal}, it follows that \(U_\lambda\cap M_{U_\lambda}\in M_{U_\lambda}\). In particular, \(\lambda\) is Fr\'echet in \(M_{U_\lambda}\). Similarly \(\gamma\) is Fr\'echet in \(M_{U_\lambda}\). Of course \(\lambda\) is a limit cardinal in \(M_{U_\lambda}\). In fact, \(\lambda\) is isolated in \(M_{U_\lambda}\). Assume not. Then \(\lambda\) is a limit of Fr\'echet cardinals in \(M_{U_\lambda}\). Therefore by \cref{LimitUniformStrong}, \(\lambda\) is a strong limit cardinal. But this contradicts that \((2^\gamma)^{M_{U_\lambda}}\geq\lambda\).

Therefore \(M_{U_\lambda}\) satisfies that \(\lambda\) is an isolated cardinal and there is a Fr\'echet cardinal \(\gamma < \lambda\) such that \(2^\gamma\geq \lambda\). But by elementarity, in \(M_{U_\lambda}\), \(j_{U_\lambda}(\lambda)\) is the least cardinal with this property. It follows that \( j_{U_\lambda}(\lambda) = \lambda\), which contradicts that \(\lambda < \sup j_{U_\lambda}(\kappa_{U_\lambda})\).
\end{proof}

Putting \cref{UpperBound} and \cref{LowerBound} together, and using \cref{SuccessorThm}, we have a rough picture of the continuum function below a nonmeasurable isolated cardinal: 

\begin{prp}[UA]
Suppose \(\lambda\) is a nonmeasurable isolated cardinal. Let \[\delta = \sup \{\gamma^+ :  \gamma < \lambda\text{ is Fr\'echet}\}\] be the strict supremum of the Fr\'echet cardinals below \(\lambda\). Then \(2^{<\delta} < \lambda \leq 2^\delta\).\qed
\end{prp}

\section{The Mitchell order without GCH}
In this section, we apply the machinery of this paper to improve the main result of \cite{MO}. We first make some general remarks to explain the presentation of the result we have chosen here.

We begin by stating some folklore facts about canonical representatives for the Rudin-Keisler equivalence classes of countably complete ultrafilters.

\begin{defn}
A countably complete ultrafilter \(U\) is {\it seed-minimal} if \([\text{id}]_U\) is the least ordinal \(\xi\) such that \(M_U = H^{M_U}(j_U[V]\cup \{\xi\})\).
\end{defn}
We include a  combinatorial reformulation without proof.
\begin{lma}
A countably complete ultrafilter \(U\) on \(\lambda\) is seed-minimal if and only if no regressive function on \(\lambda\) is one-to-one on a set in \(U\).\qed
\end{lma}

The following is immediate: 
\begin{lma}
Every countably complete ultrafilter is isomorphic to a unique seed-minimal ultrafilter.\qed
\end{lma}

\begin{defn}
Suppose \(A\) is a set. An ultrafilter \(U\) on \(P(A)\) is {\it fine} if for all \(a\in A\), for \(U\)-almost all \(\sigma \in P(A)\), \(a\in \sigma\). A fine ultrafilter \(U\) on \(P(A)\) is {\it normal} if all choice functions \(f : P(A)\to A\) are constant on set in \(U\).

A nonprincipal ultrafilter is a {\it generalized normal ultrafilter} if it is seed-minimal and isomorphic to a normal fine ultrafilter on \(P(A)\) for some set \(A\).
\end{defn}

We include a combinatorial reformulation of the notion of a generalized normal ultrafilter.

\begin{defn}
An ultrafilter \(U\) on a cardinal \(\lambda\) is {\it weakly normal} if every regressive function on \(\lambda\) takes fewer than \(\lambda\) values on a set in \(U\).
\end{defn}

\begin{thm}
If \(U\) is an ultrafilter on a cardinal \(\lambda\), the following are equivalent:
\begin{enumerate}[(1)]
\item \(U\) is weakly normal and \(j_U[\lambda]\in M_U\).
\item \(U\) is a generalized normal ultrafilter.\qed
\end{enumerate}
\end{thm}

The proof of this theorem essentially appears in \cite{MO}. It is due to Solovay when \(\lambda\) is regular and to the author when \(\lambda\) is singular.

The main theorem of this section is the following:

\begin{thm}[UA]\label{Linearity}
The Mitchell order is linear on generalized normal ultrafilters.
\end{thm}

As a corollary one can extract various corollaries about the linearity of the Mitchell order on normal fine ultrafilters; we omit these facts here since this sort of thing appears in \cite{MO}.

Our proof of \cref{Linearity} requires proving a stronger theorem:
\begin{thm}[UA]\label{IPoint}
Suppose \(U\) is a generalized normal ultrafilter and \(D\) is a countably complete uniform ultrafilter with \(D\sE U\). Then \(D\mo U\).
\end{thm}

In \cite{MO}, we proved the same theorem using a different hypothesis:
\begin{thm}\label{GCHIPoint}
Suppose \(\lambda\) is a cardinal such that \(2^{<\lambda} = \lambda\). Suppose \(U\) is a generalized normal ultrafilter on \(\lambda\) and \(D\) is a countably complete uniform ultrafilter with \(D\sE U\). Then \(D\mo U\).\qed
\end{thm}

This will be used in the proof of \cref{IPoint}. 

We also need some more facts from the general theory of the internal relation.

\begin{lma}[UA]
Suppose \(U\) is a nonprincipal countably complete ultrafilter. Let \(D\) be the \(\sE\)-least countably complete uniform ultrafilter such that \(D\not \I U\). Then for any countably complete ultrafilter \(W\), if \(W\I U\), then \(W\I D\).
\begin{proof}
We may assume without loss of generality that \(U\in \Un\). Moreover it suffices to show that for any \(W\in \Un\), if \(W\I U\) then \(W\I D\). So fix such a \(W\). To show that \(W\I D\), it suffices to show that \(t_W(D)= j_W(D)\), and for this it suffices to show that \(j_W(D)\E^{M_W} t_W(D)\). For this it is enough to show that in \(M_W\), \(t_W(D)\not \I j_W(U)\). 

Assume towards a contradiction that \(M_W\) satisfies that \(t_W(D)\I j_W(U)\). In other words, \(j^{M_W}_{t_W(D)}\) restricts to an internal ultrapower embedding of \(M_{j_W(U)}^{M_W} = j_W(M_U)\). Since \(W\I U\), it follows that \(j^{M_W}_{t_W(D)}\circ j_W\) restricts to an internal ultrapower embedding of \(M_U\). But note that \(j^{M_W}_{t_W(D)}\circ j_W = j^{M_D}_{t_D(W)}\circ j_D\). Therefore \(j_D\restriction M_U\) factors into an internal ultrapower embedding of \(M_U\). It follows easily that \(j_D\) restricts to an internal ultrapower embedding of \(M_U\), contradicting the fact that \(D\not \I U\).
\end{proof}
\end{lma}

The following key fact about internal ultrapowers is proved in \cite{RF}.

\begin{thm}[UA]
Suppose \(j_0: V\to M_0\) and \(j_1 : V\to M_1\) are ultrapower embeddings and \((i_0,i_1) : (M_0, M_1)\to N\) is their canonical comparison. Then for any ultrapower embedding \(h : N \to N'\), the following are equivalent: 
\begin{enumerate}[(1)]
\item \(h\) is an internal ultrapower embedding of \(N\).
\item \(h\) is definable over both \(M_0\) and \(M_1\).\qed
\end{enumerate}
\end{thm}

In fact, we will only use a weak corollary of this theorem:

\begin{cor}[UA]\label{tInternal}
Suppose \(D,U,W\in \Un\) and \(W\I D,U\). Then in \(M_U\), \(s_U(W)\I t_U(D)\).\qed
\end{cor}

We also use the following theorems from \cite{GCH}.

\begin{prp}[UA]\label{Partial}
Suppose \(\lambda\) is a cardinal and \(U\) is a countably complete ultrafilter such that \(M_U\) is closed under \(\lambda\)-sequences. Then for any countably complete ultrafilter \(W\) on an ordinal \(\gamma < \lambda\), \(W\in M_U\).\qed
\end{prp}

\begin{thm}[UA]\label{LocalGCH}
Suppose \(\lambda\) is a cardinal and \(\lambda^+\) carries a countably complete uniform ultrafilter. Then \(2^{<\lambda} = \lambda\).\qed
\end{thm}

We could make do without using the following theorem, proved in \cref{Counting}, but it will be convenient:

\begin{thm}[UA]\label{UltrafilterGCH}
For any cardinal \(\lambda\), \(|\Un_{\leq\lambda}| \leq (2^\lambda)^+\).\qed
\end{thm}

\begin{proof}[Proof of \cref{IPoint}]
If \(\lambda\) is a limit cardinal, then since there is a cardinal that is \(\lambda\)-supercompact, by \cite{GCH}, we have that \(2^{<\lambda} = \lambda\). We are then done by \cref{GCHIPoint}. We may therefore assume that \(\lambda\) is a successor cardinal. 

Let \(\gamma\) be the cardinal predecessor of \(\lambda\). If \(2^\gamma = \lambda\), then by \cref{GCHIPoint} we are done. We may therefore assume that \(2^\gamma > \lambda\). It follows that \(\gamma\) is a regular cardinal, since otherwise \(\gamma\) is a singular strong limit cardinal by \cref{LocalGCH}, and hence \(2^\gamma = \gamma^+\) by Solovay's theorem \cite{Solovay}.

Let \(D\) be the \(\sE\)-least ultrafilter such that \(D\not \I U\). Thus \(D\E U\). By \cref{Partial}, \(D\) is a uniform ultrafilter on \(\lambda\).

We must show \(D = U\). Let \(D' = t_U(D)\).
\begin{clm}\label{CompleteClaim}
\(D'\) is \(\lambda^+\)-complete in \(M_U\).
\end{clm}
\begin{proof}
We first show that for any \(W\in \Un^{M_U}_{\leq\lambda}\), \(M_U\) satisfies that \(W\I D'\). To see this, note that such an ultrafilter \(W\) satisfies \(W\I U\) since \(j_W\restriction M_U = j_W^{M_U}\) by the closure of \(M_U\) under \(\lambda\)-sequences. It follows from \cref{tInternal} that \(M_U\) satisfies \(s_U(W) \I D'\), but this is equivalent to the fact that \(M_U\) satisfies \(W\I D'\).

Let \(U'\) be the generalized normal ultrafilter on \(\gamma\) derived from \(U\). Then \(U'\in M_U\) by \cref{Partial}. By an argument due to Solovay, it follows that in \(M_U\), every set \(A\subseteq P(\gamma)\) is in the ultrapower of a countably complete ultrafilter on \(\gamma\). Thus in \(M_U\), \(2^{2^\gamma} = |\Un_{\leq\gamma}| = (2^\gamma)^+\) by \cref{UltrafilterGCH}.

Note that \((\lambda^\sigma)^{M_U} > \lambda^+\): otherwise \((\lambda^\sigma)^{M_U} = \lambda^+\), so applying \cref{LocalGCH} in \(M_U\), we have \((2^{<\lambda})^{M_U} = \lambda\), and so since \(P(\lambda)\subseteq M_U\), in actuality \(2^{<\lambda} = \lambda\), contrary to assumption.

Therefore by \cref{SuccessorLemma}, \(M_U\) satisfies that \(\lambda^\sigma\) is isolated. We can apply \cref{UpperBound} at the Fr\'echet cardinal \(\gamma\) to conclude that \(M_U\) satisfies \(2^\gamma < \lambda^\sigma\) and hence \((2^\gamma)^+ < \lambda^\sigma\). Since \(2^\gamma \geq \lambda^+\) and \(P(\gamma)\subseteq M_U\), \(M_U\) satisfies that \(2^\gamma \geq \lambda^+\). Thus \(M_U\) satisfies \[2^{(\lambda^+)} \leq 2^{2^\gamma} = (2^\gamma)^+ < \lambda^\sigma\]

Now in \(M_U\), \(\lambda^+\) is a non-Fr\'echet regular cardinal and \(2^{(\lambda^+)} < \lambda^\sigma = (\lambda^+)^\sigma\). Therefore by \cref{UpperBoundLemma}, it follows that \(D'\) is \(\lambda^\sigma\)-complete in \(M_U\), which proves the claim.
\end{proof}
Given the claim, we finish the proof as follows. Since \(\lambda\) is regular, the weak normality of \(U\) implies that \([\text{id}]_U = \sup j_U[\lambda]\). Since \(D'\) is \(\lambda^+\)-complete, \(j^{M_U}_{D'}\) is continuous at \(\sup j_U[\lambda]\). Thus \[j^{M_U}_{D'}([\text{id}]_U) = \sup j^{M_U}_{D'}\circ j_U[\lambda] = \sup j^{M_D}_{t_D(U)}\circ j_D[\lambda] \leq j^{M_D}_{t_D(U)}([\text{id}]_D)\] The final inequality comes from the fact that \(\sup j_D[\lambda] \leq [\text{id}]_D\); this holds because \(D\) is a uniform ultrafilter on \(\lambda\). Thus \(U\E D\), and hence \(U = D\), as desired.
\end{proof}

\section{Strongly tall cardinals and set-likeness}
There is a natural question left open by the results of \cite{SC} that we resolve here:
\begin{thm}[UA]\label{STSC}
Suppose \(\kappa\) is the least ordinal such that for all ordinals \(\alpha\), there is some ultrapower embedding \(j: V\to M\) such that \(j(\kappa) >\alpha\). Then \(\kappa\) is supercompact.\qed
\end{thm}

What we show here is the following:

\begin{prp}[UA]\label{StronglyTall}
Suppose \(\kappa\) is the least ordinal such that for all ordinals \(\alpha\), there is some ultrapower embedding \(j: V\to M\) such that \(j(\kappa) >\alpha\). Then \(\kappa\) is \(\omega_1\)-strongly compact.
\end{prp}

Clearly \(\kappa\) is then the least \(\omega_1\)-strongly compact cardinal. It follows that \(\kappa\) is supercompact by the following theorem from \cite{SC}:
\begin{thm}[UA]
The least \(\omega_1\)-strongly compact cardinal is supercompact.\qed
\end{thm}

For the proof of \cref{StronglyTall}, we use the following fact:
\begin{lma}[UA]\label{FixInternal}
Suppose \(\xi\) is an ordinal and \(W\) is the \(\sE\)-least uniform countably complete ultrafilter \(Z\) such that \(j_Z(\xi) \neq \xi\). Then for any countably complete ultrafilter \(U\) such that \(j_U(\xi) = \xi\), \(U\I W\).
\begin{proof}
By \cref{InternalTranslation}, it suffices to show that in \(M_U\), \(j_U(W) \E \tr U W\). For this it is enough to show that \(j^{M_U}_{\tr U W}(j_U(\xi)) \neq j_U(\xi)\), since in \(M_U\), \(j_U(W)\) is the \(\sE\)-least countably complete ultrafilter \(Z\) such that \(j_Z^{M_U}(\xi) \neq \xi\). This is a consequence of the following calculation: \begin{align*}j^{M_U}_{\tr U W}(j_U(\xi)) &= j^{M_W}_{\tr W U}(j_W(\xi))\geq j_W(\xi) > \xi = j_U(\xi) \qedhere\end{align*}
\end{proof}
\end{lma}

\begin{proof}[Proof of \cref{StronglyTall}]
We show that for any ordinal \(\alpha \geq \kappa\), \(\alpha^\sigma = \alpha^+\). Let \(\lambda = \alpha^\sigma\). Let \(U = U_\lambda\). Let \(\xi\) be the least fixed point of \(j_{U}\) such that \(\xi > \textsc{crt}(j_{U})\). Let \(W\) be the \(\sE\)-least uniform countably complete ultrafilter such that \(j_W(\xi) \neq \xi\). By \cref{FixInternal}, \(U \I W\). 

We claim \(W\not \I U\). Suppose towards a contradiction that \(W\I U\). Then by \cref{Commutativity}, \(j_W(j_U\restriction \text{Ord}) = j_U\restriction \text{Ord}\). This is a contradiction since \(\xi\) is \(\Delta_1\)-definable from \(j_U\restriction \text{Ord}\) without parameters, yet \(j_W(\xi) \neq \xi\).

Thus \(U\I W\) and \(W\not \I U\). By \cref{Nonisolation2}, it follows that \(\lambda\) is not an isolated cardinal. Thus \(\lambda = \alpha^+\), as desired.
\end{proof}

We connect this up with some ideas from \cite{IR}.

\begin{defn}
A {\it pointed ultrapower} is a pair \((M,\alpha)\) where \(M\) is an ultrapower of \(V\) and \(\alpha\) is an ordinal. The class of pointed ultrapowers is denoted \(\mathscr P\). If \(\mathcal M = (M,\alpha)\) is a pointed ultrapower, then \(\alpha_\mathcal M\) denotes \(\alpha\).
\end{defn}

We will sometimes abuse notation by writing \(\mathcal M\) when we really mean the ultrapower \(M\) such that \(\mathcal M= (M,\alpha_\mathcal M)\).

\begin{defn}
The {\it completed seed order} is defined on \(\mathcal M_0,\mathcal M_1\in \mathscr P\) by setting \(\mathcal M_0 \swo \mathcal M_1\) if there is a pair of elementary embeddings \((i_0,i_1) : (\mathcal M_0,\mathcal M_1)\to N\) such that \(i_1\) is an internal ultrapower embedding of \(M_1\) and \(i_0(\alpha_{\mathcal M_0}) < i_1(\alpha_{\mathcal M_1})\).
\end{defn}

Not every pointed ultrapower must have a rank in the completed seed order, but for those that do, we use the following notation:

\begin{defn}For any \(\mathcal M\in \mathscr P\), \(|\mathcal M|_\infty\) denotes the rank of \(\mathcal M\) in the completed seed order if it exists.\end{defn}

Here we will only consider the completed seed order assuming UA, in which case we have the following fact from \cite{IR}:

\begin{lma}[UA]
Suppose \(\mathcal M_0,\mathcal M_1\in \mathscr P\). Then the following are equivalent:
\begin{enumerate}[(1)]
\item \(\mathcal M_0\swo \mathcal M_1\).
\item There exist internal ultrapower embeddings \((i_0,i_1) : (\mathcal M_0,\mathcal M_1)\to N\) such that \(i_0(\alpha_{\mathcal M_0}) < i_1(\alpha_{\mathcal M_1})\).
\end{enumerate}
\end{lma}

We have the following fact about \(\mathscr P\):

\begin{lma}\label{PointedAbs0}
Suppose \(N\) is an ultrapower. Then \(\mathscr P^{N}\subseteq\mathscr P\) and \({\swo^N}\) is a suborder of \({\swo} \restriction \mathscr P^N\).\qed
\end{lma}

\begin{cor}
For any ultrapower \(N\) and \(\mathcal M\in \mathscr P^N\), \(|\mathcal M|_\infty^N\leq |\mathcal M|_\infty\).
\end{cor}

With UA, we can prove sharper results:

\begin{lma}[UA]\label{PointedAbs}
Suppose \(N\) is an ultrapower. Then \(\mathscr P^{N}\subseteq\mathscr P\) and \({\swo^N}\) is equal to \({\swo} \restriction \mathscr P^N\). Moreover, for any \(\mathcal M\in \mathscr P\), there is some \(\mathcal M'\in \mathscr P^N\) such that \(\mathcal M \equiv_S \mathcal M'\).\qed
\end{lma}

\begin{cor}[UA]
For any ultrapower \(N\) and \(\mathcal M\in \mathscr P^N\), \(|\mathcal M|_\infty^N= |\mathcal M|_\infty\).\qed
\end{cor}

\begin{lma}[UA]
Suppose \(\alpha\) is an ordinal. Then the following are equivalent:
\begin{enumerate}[(1)]
\item \(\alpha = |(V,\xi)|_\infty\) for some ordinal \(\xi\).
\item \(\alpha\) is fixed by all ultrapower embeddings.
\end{enumerate}
\begin{proof}
Assume (1). Let \(i : V\to N\) be an ultrapower embedding. Then \(i(\alpha) = |(N,i(\xi))|_\infty^N\). By \cref{PointedAbs0}, \( |(N,i(\xi))|_\infty^N \leq |(N,i(\xi))|_\infty\). It is easy to see that every predecessor of \((N,i(\xi))\) in the completed seed order is a predecessor of \((V,\xi)\) in the completed seed order. Therefore \(|(N,i(\xi))|_\infty \leq |(V,\xi)|_\infty = \alpha\). Putting everything together \(i(\alpha) \leq \alpha\), so \(i(\alpha) = \alpha\). Since \(i\) was an arbitrary ultrapower embedding, (2) holds.

It is to prove (2) implies (1) that we need UA. Fix an ultrapower embedding \(j : V\to M\) such that for some \(\xi'\), \(\alpha\) is the rank of \((M,\xi')\) in the completed seed order. Then by \cref{PointedAbs}, in \(M\), \(\alpha\) is the rank of \((M,\xi')\) in the completed seed order. But \(\alpha = j(\alpha)\) since \(\alpha\) is fixed by all ultrapower embeddings. It follows by the elementarity of \(j : V\to M\) that there is some \(\xi\) such that \(\alpha\) is the rank of \((V,\xi)\) in the completed seed order, as claimed.
\end{proof}
\end{lma}

\begin{cor}[UA]\label{FixRank}
The following are equivalent:
\begin{enumerate}[(1)]
\item For all \(\mathcal M\in \mathscr P\), \(|\mathcal M|_\infty\) exists.
\item For unboundedly many ordinals \(\alpha\), \(\alpha\) is fixed by all ultrapower embeddings.
\end{enumerate}
\begin{proof}
That (1) implies (2) is immediate. To show (2) implies (1), note that (2) implies that there are unboundedly many ordinals \(\xi\) such that the rank of \((V,\xi)\) in the completed seed order is an ordinal. Therefore for all ordinals \(\xi\), \((V,\xi)\) has a rank in the completed seed order, since \((V,\xi_0) \swo (V,\xi_1)\) when \(\xi_0 < \xi_1\). But for any \(\mathcal M\in \mathscr P\), there is some \(\xi\) such that \(\mathcal M\swo (V,\xi)\). Therefore \(\mathcal M\) has a rank in the completed seed order, i.e. \(|\mathcal M|_\infty\) exists.
\end{proof}
\end{cor}

\begin{lma}\label{FixTall}
Exactly one of the following holds:
\begin{enumerate}[(1)]
\item For unboundedly many ordinals \(\alpha\), \(\alpha\) is fixed by all ultrapower embeddings.
\item There is an ordinal \(\kappa\) such that for all ordinals \(\alpha\), for some ultrapower embedding \(j : V\to M\), \(j(\kappa) > \alpha\).
\end{enumerate}
\begin{proof}
Obviously (2) implies (1) fails, so we just need to show that if (2) fails then (1) holds. Assume (2) fails. Fix an ordinal \(\xi\). We will define an ordinal \(\alpha\geq \xi\) that is fixed by all ultrapower embeddings. Let \[\alpha = \sup \{j(\xi) : j\text{ is an ultrapower embedding of \(V\)}\}\] The supremum \(\alpha\) exists since (2) fails. But for any ultrapower embedding \(i : V\to N\), 
\begin{align*}i(\alpha) &= \sup \{j(i(\xi)) : j\text{ is an internal ultrapower embedding of \(N\)}\}\\
&\leq \sup \{j(\xi) : j\text{ is an ultrapower embedding of \(V\)}\}\\
&=\alpha
\end{align*}
Therefore \(\alpha\) is fixed by all ultrapower embeddings.  
\end{proof}
\end{lma}

Combining \cref{STSC}, \cref{FixRank}, and \cref{FixTall}, we obtain the following fact:
\begin{thm}[UA]
Exactly one of the following holds:
\begin{enumerate}[(1)]
\item For all \(\mathcal M\in \mathscr P\), \(|\mathcal M|_\infty\) exists.
\item There is a supercompact cardinal.\qed
\end{enumerate}
\end{thm}
With a bit more work, one can show:
\begin{thm}[UA]
Suppose \(M\) is an ultrapower and \(\alpha\) is an ordinal less than the least supercompact cardinal of \(M\). Then \(|(M,\alpha)|_\infty\) exists.\qed
\end{thm}

\section{Counting countably complete ultrafilters}\label{Counting}
In this final section we prove a cardinal arithmetic fact whose proof requires techniques from this paper.

\begin{thm}[UA]\label{NumberMeas}
A set \(X\) carries at most \((2^{|X|})^+\) countably complete ultrafilters.
\end{thm}

\cref{NumberMeas} is of course a consequence of GCH since \(X\) carries at most \(2^{2^{|X|}}\) ultrafilters, but we see no way to obtain it as a corollary of the theorems in \cite{GCH}. Instead the proof given here imitates that of \cite{GCH} Lemma 4.3, replacing the Mitchell order with more general concepts.

\begin{defn}
If \(U\) and \(U'\) are countably complete ultrafilters on ordinals \(\gamma\) and \(\gamma'\), we set \(U\sE U'\) if there is a sequence \(\langle U_\alpha : \alpha < \gamma'\rangle\) such that
\begin{enumerate}[(1)]
\item For \(U'\)-almost all \(\alpha < \gamma'\), \(U_\alpha\) is a countably complete ultrafilter on \(\alpha\).
\item For any \(X\subseteq \gamma\), \(X\in U\) if and only if \(X\cap \alpha\in U_\alpha\) for \(U'\)-almost all \(\alpha < \gamma'\).
\end{enumerate}
\end{defn}

\begin{defn}
An ultrafilter \(U\) on an ordinal \(\alpha\) is {\it tail uniform} (or just {\it uniform}) if \(\alpha \setminus \beta\in U\) for all \(\beta < \alpha\). 
\end{defn}
\begin{defn}
For any ordinal \(\alpha\), \(\Un_\alpha\) denotes the set of uniform countably complete ultrafilters on \(\alpha\), \(\Un_{<\alpha} = \bigcup_{\beta <\alpha} \Un_{\beta}\), \(\Un_{\leq\alpha}  = \bigcup_{\beta \leq\alpha} \Un_{\beta}\), and \(\Un = \bigcup_{\alpha\in \text{Ord}} \Un_\alpha\).
\end{defn}

\begin{defn}
For \(U\in \Un\), let \(|U|_S\) denote the rank of \(U\) in \(\sE\) restricted to \(\Un\).
\end{defn}

We use the following lemma for our main result. The proof appears in \cite{IR}.

\begin{lma}[UA]\label{UnFix}
Suppose \(U\) is a nonprincipal countably complete uniform ultrafilter. Then \(j_U(|U|_S) > |U|_S\).\qed
\end{lma}

The following bound is essentially immediate from the definition of \(\sE\).

\begin{lma}\label{PredTriv}
A countably complete uniform ultrafilter on \(\lambda\) has at most \(\prod_{\alpha < \lambda} |\Un_{\leq\alpha}|\) uniform \(\sE\)-predecessors.\qed
\end{lma}

\begin{defn}
An ultrafilter \(U\) on a set \(X\) is {\it Fr\'echet uniform} if for all \(A\subseteq X\) with \(|A| < |X|\), \(X\setminus A\in U\). A cardinal \(\lambda\) is {\it Fr\'echet} if it carries a countably complete Fr\'echet uniform ultrafilter.
\end{defn}

\begin{thm}[UA]\label{FrechCount}
Suppose \(\lambda\) is a Fr\'echet cardinal. Then for any \(\alpha < \lambda\), \(|\Un_{\leq\alpha}| \leq 2^\lambda\).
\begin{proof}
We may assume by induction that the theorem holds below \(\lambda\). If \(\lambda\) is a limit of Fr\'echet cardinals, then the inductive hypothesis easily implies the conclusion of the theorem. Therefore assume \(\lambda\) is not a limit of Fr\'echet cardinals. 

Let \(U_\lambda\) be the \(\sE\)-least Fr\'echet uniform ultrafilter on \(\lambda\). Since \(\lambda\) is not a limit of uniform cardinals, for every ultrafilter \(W\) on a cardinal less than \(\lambda\), \(W\I U_\lambda\).

Assume towards a contradiction that \(\gamma < \lambda\) is least such that \(|\Un_{\leq\gamma}| > 2^\lambda\). We record that since \(|\Un_{\leq \gamma}|\geq \kappa_{U_\lambda}\), in fact \(\gamma \geq \kappa_{U_{\lambda}}\). 

Let \(\alpha = \sup \{|W|_S :\textsc{sp}(W) < \gamma\}\). We claim that for any ordinal \(\xi\in [\alpha,(2^\lambda)^+]\), there is some \(D\) such that \(j_D(\xi) > \xi\). To see this let \(D\) be the unique countably complete uniform ultrafilter with \(|D|_S = \xi\). Then \(D\) is a uniform ultrafilter on \(\gamma\), so \(D\) is nonprincipal. Therefore by \cref{UnFix}, \(j_D(\xi) > \xi\) as desired.

Let \(\xi\) be the least ordinal above \(\alpha\) fixed by \(j_{U_{\lambda}}\) and also by \(j_D\) for all \(D\in \Un_{<\gamma}\). Then \(\xi < (2^\lambda)^+\) since the intersection of \(2^\lambda\)-many \(\omega\)-club subsets of \((2^\lambda)^+\) is \(\omega\)-club. Let \(W\in \Un\) be the \(\sE\)-least ultrafilter with \(j_W(\xi) > \xi\). By \cref{UnFix}, \(W\) is a countably complete uniform ultrafilter on \(\gamma < \lambda\), so \[W\I U_{\lambda}\] by \cref{SuccessorInternal} since \(\lambda\) is not a limit of Fr\'echet cardinals. Moreover, since \(j_{U_{\lambda}}(\xi) = \xi\), \[U_{\lambda}\I W\] by \cref{FixInternal}.

Since \(\kappa_W \leq \gamma < \lambda\), \[\Un_{{<}\kappa_W}\I U_\lambda\] by \cref{SuccessorInternal}. Since \(\kappa_{U_\lambda}\leq \lambda\) is a strong limit cardinal while \(2^{2^\gamma} \geq |\Un_{\leq\gamma}| > 2^\lambda\), \(\kappa_{U_\lambda} \leq \gamma\) and therefore \[\Un_{{<}\kappa_{U_\lambda}}\I W\] by \cref{FixInternal}. Since \(W\) and \(U_\lambda\) are nonprincipal, this contradicts \cref{Noncommutative}.
\end{proof}
\end{thm}

As a corollary we can prove \cref{NumberMeas}:
\begin{proof}[Proof of \cref{NumberMeas}]
It clearly suffices to show that \(|\Un_{\leq\lambda}| \leq (2^\lambda)^+\) for all cardinals \(\lambda\). By induction assume that \(|\Un_{\leq \bar \lambda}| \leq (2^{\bar \lambda})^+\) for all cardinals \(\bar \lambda < \lambda\).
 
Assume first that \(\lambda\) is not Fr\'echet. Then any countably complete ultrafilter on \(\lambda\) concentrates on a set \(A\in [\lambda]^{<\lambda}\). Therefore there is a surjection from the set of pairs \((A,U)\) where \(A\in [\lambda]^{<\lambda}\) and \(U\) is a countably complete ultrafilter on \(A\) to \(\Un_{\leq\lambda}\). The number of such pairs is bounded by \[[\lambda]^{<\lambda}\cdot \sup_{\bar \lambda < \lambda} |\Un_{\leq\bar \lambda}|\leq 2^\lambda \cdot (2^\lambda)^+\]
Thus \(|\Un_{\leq\lambda}|\leq (2^\lambda)^+\).

Now assume that \(\lambda\) is Fr\'echet. We claim that for any ultrafilter \(U\in \Un_{\leq\lambda}\) has at most \(2^\lambda\)-many \(\sE\)-predecessors. By \cref{PredTriv}, \(U\) has at most \(\prod_{\alpha < \lambda} |\Un_{\leq\alpha}|\) predecessors. But by \cref{FrechCount}, for all \(\alpha < \lambda\), \(|\Un_{\leq\alpha}| \leq 2^\lambda\), so \[\prod_{\alpha < \lambda} |\Un_{\leq\alpha}| \leq (2^\lambda)^\lambda = 2^\lambda\]
Thus \(U\) has at most \(2^\lambda\)-many \(\sE\)-predecessors. Since \(\Un_{\leq\lambda}\) is wellordered by \(\sE\) with initial segments of length at most \(2^\lambda\), \(|\Un_{\leq\lambda}|\leq (2^\lambda)^+\).
\end{proof}

We remark that it is not hard to use this fact to prove that GCH holds above the least strongly compact cardinal. This proof does not yield a result that is as local as the one in \cite{GCH}.

\bibliography{Bibliography}{}

\begin{thebibliography}{10}

\bibitem{SC}
Gabriel Goldberg.
\newblock The equivalence of strong and supercompactness under {UA}.
\newblock To appear.

\bibitem{MO}
Gabriel Goldberg.
\newblock The linearity of the {M}itchell order.
\newblock To appear.

\bibitem{GCH}
Gabriel Goldberg.
\newblock Strongly compact cardinals and the {GCH}, revisited.
\newblock To appear.

\bibitem{Prikry}
Kenneth Kunen and Karel Prikry.
\newblock On descendingly incomplete ultrafilters.
\newblock {\em J. Symbolic Logic}, 36:650--652, 1971.

\bibitem{Ketonen}
Jussi Ketonen.
\newblock Strong compactness and other cardinal sins.
\newblock {\em Ann. Math. Logic}, 5:47--76, 1972/73.

\bibitem{MagidorBagaria}
Joan Bagaria and Menachem Magidor.
\newblock On {$\omega_1$}-strongly compact cardinals.
\newblock {\em J. Symb. Log.}, 79(1):266--278, 2014.

\bibitem{SO}
Gabriel Goldberg.
\newblock The seed order.
\newblock To appear.

\bibitem{IR}
Gabriel Goldberg.
\newblock The internal relation.
\newblock To appear.

\bibitem{Kunen}
Kenneth Kunen.
\newblock Elementary embeddings and infinitary combinatorics.
\newblock {\em J. Symbolic Logic}, 36:407--413, 1971.

\bibitem{Kanamori}
Robert~M. Solovay, William~N. Reinhardt, and Akihiro Kanamori.
\newblock Strong axioms of infinity and elementary embeddings.
\newblock {\em Ann. Math. Logic}, 13(1):73--116, 1978.

\bibitem{RF}
Gabriel Goldberg.
\newblock The {U}ltrapower {A}xiom and the {R}udin-{F}rolik order.
\newblock To appear.

\bibitem{Solovay}
Robert~M. Solovay.
\newblock Strongly compact cardinals and the {GCH}.
\newblock In {\em Proceedings of the {T}arski {S}ymposium ({P}roc. {S}ympos.
  {P}ure {M}ath., {V}ol. {XXV}, {U}niv. {C}alifornia, {B}erkeley, {C}alif.,
  1971)}, pages 365--372. Amer. Math. Soc., Providence, R.I., 1974.

\end{thebibliography}
\bibliographystyle{unsrt}

\end{document}